\documentclass[12pt]{amsart}
\usepackage{amsmath,amssymb,mathrsfs,bm}
\usepackage{pstricks}
\usepackage[colorlinks=true,linkcolor=black,citecolor=black,urlcolor=black]{hyperref}

\psset{unit=1pt}
\psset{arrowsize=4pt 1}
\psset{linewidth=.5pt}

\newgray{lightgray}{.80}

\countdef\x=23
\countdef\y=24
\countdef\z=25
\countdef\t=26

\def\graybox(#1,#2){
\x=#1 \y=#2 
\z=\x \t=\y
\advance\z by 10 
\advance\t by 10 
\psframe[fillstyle=solid,fillcolor=lightgray,linewidth=0pt](\x,\y)(\z,\t) 
\psline[linewidth=.5pt](\x,\y)(\x,\t)(\z,\t)(\z,\y)(\x,\y)}

\def\emptygraybox(#1,#2){
\x=#1 \y=#2 
\z=\x \t=\y
\advance\z by 10 
\advance\t by 10 
\psframe[fillstyle=solid,fillcolor=lightgray,linewidth=0pt,linecolor=lightgray](\x,\y)(\z,\t)}

\def\blankbox(#1,#2){
\x=#1 \y=#2 
\z=\x \t=\y
\advance\z by 10 
\advance\t by 10 
\psframe[linewidth=.5pt](\x,\y)(\z,\t)}

\def\whitebox(#1,#2){
\x=#1 \y=#2 
\z=\x \t=\y
\advance\z by 10 
\advance\t by 10 
\psframe[fillstyle=solid,fillcolor=white,linewidth=0pt](\x,\y)(\z,\t) 
\psline[linewidth=.5pt](\x,\y)(\x,\t)(\z,\t)(\z,\y)(\x,\y)}

\def\whiteboxb(#1,#2){
\x=#1 \y=#2 
\z=\x \t=\y
\advance\z by 10 
\advance\t by 10 
\psframe[fillstyle=solid,fillcolor=white,linewidth=0pt](\x,\y)(\z,\t)}

\newcommand{\define}{\textbf}

\newcommand{\excise}[1]{}

\newcommand{\ul}{\underline}

\renewcommand{\setminus}{\smallsetminus}
\renewcommand{\phi}{\varphi}

\renewcommand{\tilde}{\widetilde}

\renewcommand{\bar}{\overline}

\newcommand{\Ess}{\mathscr{E}\hspace{-.4ex}ss}

\newcommand{\bk}{\mathbf{k}}
\newcommand{\bp}{\mathbf{p}}
\newcommand{\bq}{\mathbf{q}}

\newcommand{\triple}{{\bm\tau}}

\newtheorem{theorem}{Theorem}[section]
\newtheorem{lemma}[theorem]{Lemma}
\newtheorem{proposition}[theorem]{Proposition}
\newtheorem{corollary}[theorem]{Corollary}

\newtheorem*{thm*}{Theorem}
\newtheorem*{lem*}{Lemma}
\newtheorem*{prop*}{Proposition}
\newtheorem*{cor*}{Corollary}

\theoremstyle{definition}

\newtheorem{definition}[theorem]{Definition}
\newtheorem{remark}[theorem]{Remark}

\newtheorem*{defn*}{Definition}
\newtheorem*{rmk*}{Remark}

\begin{document}

\title{Vexillary signed permutations revisited}

\date{June 1, 2016}

\author{David Anderson}
\address{Department of Mathematics, The Ohio State University, Columbus, OH 43210}
\email{anderson.2804@math.osu.edu}

\author{William Fulton}
\address{Department of Mathematics,
University of Michigan,
Ann Arbor, Michigan  48109-1043, U.S.A.}
\email{wfulton@umich.edu}

\begin{abstract}
We study the combinatorial properties of vexillary signed permutations, which are signed analogues of the vexillary permutations first considered by Lascoux and Sch\"utzenberger.  We give several equivalent characterizations of vexillary signed permutations, including descriptions in terms of essential sets and pattern avoidance, and we relate them to the vexillary elements introduced by Billey and Lam.
\end{abstract}

\thanks{DA was partially supported by NSF Grant DMS-1502201 and a postdoctoral fellowship from the Instituto Nacional de Matem\'atica Pura e Aplicada (IMPA)}

\maketitle

\section*{Introduction}

The class of {\em vexillary permutations} in $S_n$, first identified by Lascoux and Sch\"utzenberger \cite{ls,ls-lrr}, plays a central role in the combinatorics of the symmetric group and the corresponding geometry of Schubert varieties and degeneracy loci.  The name derives from the fact that the Schubert polynomial of a vexillary permutation is equal to a flagged Schur polynomial.  This was given a geometric explanation in \cite{fulton}: vexillary permutations correspond to degeneracy loci defined by simple rank conditions, whose classes are computed by variations of the Kempf-Laksov determinantal formula.

Degeneracy loci of other classical types are indexed by the group $W_n$ of signed permutations.  In the course of proving analogous Pfaffian formulas for such loci \cite{af1,af11}, we constructed {\em vexillary signed permutations}, starting with the notion of a {\em triple}.  A triple is three $s$-tuples of positive integers, $\triple = (\bk,\bp,\bq)$, with $\bk = (0< k_1< \cdots <k_s)$, $\bp = (p_1 \geq \cdots \geq p_s > 0)$, and $\bq = (q_1\geq \cdots \geq q_s >0)$, satisfying $k_{i+1}-k_i \leq p_i-p_{i+1}+q_i-q_{i+1}$ for $1\leq i\leq s-1$.  Given such a triple, one constructs a signed permutation $w=w(\triple)$ (see \S\ref{s.triples} for the details).  By definition, our vexillary signed permutations are the ones arising this way.

Vexillary permutations in $S_n$ have many equivalent characterizations, some of which will be reviewed below; others may be found in \cite{macdonald}.  The quickest one is via {\em pattern avoidance}: a permutation $v$ is vexillary if and only if it avoids the pattern $[2\;1\;4\;3]$---that is, there are no indices $a<b<c<d$ such that $v(b)<v(a)<v(d)<v(c)$.  Another is that the Stanley symmetric function of a vexillary permutation is equal to a single Schur function.  The latter property was taken as the starting point for Billey and Lam's extension of ``vexillary'' to other Lie types: they defined three distinct classes of vexillary elements in types B, C, and D, whose Stanley functions are equal to single Schur $P$- or $Q$-functions \cite{bl}.  Our starting point is the geometric property: the vexillary signed permutations considered in \cite{af1,af11} correspond to degeneracy loci defined by rank conditions of a particularly simple kind, and whose Schubert polynomials can be written as flagged Pfaffians.  As often happens, properties that coincide in type A diverge in other types---it turns out that Billey and Lam's type B vexillary elements are the same as our vexillary signed permutations, which do not depend on type.
\footnote{Since a motivating property of our vexillary signed permutations is the flagged Pfaffian formula for Schubert polynomials, they could be called {\it Pfaffian-vexillary}, when it is necessary to distinguish them from Billey and Lam's vexillary elements.}

The main goal of this article is to provide alternative characterizations of vexillary signed permutations.  In addition to a signed pattern avoidance criterion, we give characterizations of vexillary signed permutations in terms of essential sets of rank conditions, embeddings in symmetric groups, and Stanley symmetric functions (the latter as a consequence of the coincidence with Billey-Lam's type B vexillary elements).  

The definition of a vexillary signed permutation $w(\triple)$, given combinatorially in \S\ref{s.triples}, comes with a geometric explanation in terms of degeneracy loci.  Given an odd-rank vector bundle $V$ on a variety $X$, equipped with a nondegenerate quadratic form and flags of isotropic subbundles $V \supset E_1 \supset E_2 \supset \cdots$ and $V \supset F_1 \supset F_2 \supset \cdots$, a signed permutation $w$ determines a degeneracy locus $\Omega_w \subseteq X$, defined by imposing certain rank conditions $\dim(E_p \cap F_q) \geq k$.  Given a triple $\triple=(\bk,\bp,\bq)$, the vexillary signed permutation $w(\triple)$ is defined so that the rank conditions for the corresponding degeneracy locus are $\dim(E_{p_i}\cap F_{q_i}) \geq k_i$, for $1\leq i\leq s$, and so that $w(\triple)$ is minimal (in Bruhat order) with this property.  The inequalities required on $\bk,\bp,\bq$ guarantee that these rank conditions are feasible: they come from the inclusion of vector spaces
\[
  (E_{p_{i+1}} \cap F_{q_{i+1}})/(E_{p_i} \cap F_{q_i}) \subseteq E_{p_{i+1}}/E_{p_i} \oplus F_{q_{i+1}}/F_{q_i}.
\]
There is an equivalent interpretation for even-rank vector bundles with a symplectic form.

For degeneracy loci corresponding to ordinary permutations, a minimal list of non-redundant rank conditions is determined by the {\em essential set} of the permutation \cite{fulton}.  The essential set defined in \cite{fulton} consists of pairs $(p,q)$ appearing in rank conditions $\dim(E_p\cap F_q) \geq k$; the value of $k$ is determined by a rank function associated to the permutation.  Here we will modify the definitions slightly, for both ordinary and signed permutations, to include the value of the rank function.  Our essential sets, introduced and studied in \cite{a-diag}, consist of certain ``basic triples'' $(k,p,q)$; the pairs $(p,q)$ will be called {\em essential positions}.  (Precise definitions are reviewed in \S\ref{s.diagrams}.)

One characterization of vexillary signed permutations is that the corresponding essential positions $(p_i,q_i)$ may be ordered so that $p_1\geq \cdots \geq p_s >0$ and $q_1\geq \cdots \geq q_s >0$.  Another is given via comparison with symmetric groups: the definition of the group of signed permutations leads naturally to an embedding $\iota\colon W_n\hookrightarrow S_{2n+1}$, and remarkably, a signed permutation $w$ is vexillary if and only if $\iota(w)$ is.  (The same statement applies to another natural embedding $\iota'\colon W_n \hookrightarrow S_{2n}$.)  We also give a characterization via {\em signed pattern avoidance}, analogous to the pattern avoidance criterion for $S_n$; see \S\ref{s.patterns} for details.

\begin{thm*}
Let $w$ be a signed permutation.  The following are equivalent:
\begin{enumerate}
\item $w$ is vexillary, i.e., it is equal to $w(\triple)$ for some triple $\triple$. \label{tcond.def}

\smallskip

\item The essential positions of $w$ can be ordered $(p_1,q_1), \ldots, (p_s,q_s)$, so that $p_1 \geq \cdots \geq p_s > 0$ and $q_1 \geq \cdots \geq q_s >0$. \label{tcond.ess}

\smallskip

\item $\iota(w)$ is vexillary, as a permutation in $S_{2n+1}$. \label{tcond.perm}

\smallskip

\item $\iota'(w)$ is vexillary, as a permutation in $S_{2n}$. \label{tcond.perm-even}

\smallskip

\item  $w$ avoids the nine signed patterns $[2\;1]$, $[\bar{3}\;2\;\bar{1}]$, $[\bar{4}\;\bar{1}\;\bar{2}\;3]$, $[\bar{4}\;1\;\bar{2}\;3]$, $[\bar{3}\;\bar{4}\;\bar{1}\;\bar{2}]$, $[\bar{3}\;\bar{4}\;1\;\bar{2}]$, $[\bar{2}\;\bar{3}\;4\;\bar{1}]$, $[2\;\bar{3}\;4\;\bar{1}]$, and $[3\;\bar{4}\;\bar{1}\;\bar{2}]$. \label{tcond.pattern}

\end{enumerate}
\end{thm*}

\noindent
The theorem is a consequence of Theorem~\ref{t.vex-characterization} combined with Proposition~\ref{p.patterns} and Remark~\ref{r.even}.

A triple $\triple$ determines a strict partition $\lambda(\triple)$ of length $s$, by setting $\lambda_{k_i} = p_i+q_i-1$, and filling in the remaining $\lambda_k$ minimally subject to the strict decreasing requirement; these partitions index the multi-Schur Pfaffians appearing as degeneracy locus formulas in \cite{af1}.  In \S\ref{s.lyd}, we use ``labellings'' of their (shifted) Young diagrams to describe how to move between vexillary signed permutations along chains in Bruhat order.  We also show that the sum of the parts of $\lambda(\triple)$ computes the length of $w(\triple)$ (Corollary~\ref{c.ydl}).

The pattern avoidance criterion \eqref{tcond.pattern} coincides with that of Billey and Lam in type B, and as a consequence, we can add another equivalent condition to the list in the above theorem: writing $H_w$ for the Stanley symmetric function of type B (see, e.g., \cite{bh} or \cite[(6.3)]{fk}), we have
\begin{enumerate}\setcounter{enumi}{5}
\item {\it $H_w$ is equal to a Schur $P$-function $P_{\lambda}$.}\label{tcond.stanley}
\end{enumerate}
(In fact, Corollary~\ref{c.stanley} says $H_{w(\triple)} = P_{\lambda(\triple)}$.)  The equivalence of \eqref{tcond.stanley} with \eqref{tcond.pattern} is the content of \cite[Theorem~14]{bl}.  Billey and Lam also give pattern avoidance criteria for determining when other variants of Stanley symmetric functions are equal to $P$- and $Q$-functions, but our vexillary signed permutations form a strict subset of the those defined in \cite{bl} for types C and D.

Is there a geometric property enjoyed by Billey-Lam's vexillary elements of type C or D?  In \cite{af11}, we find ``flagged theta-polynomial'' formulas for a larger class of signed permutations constructed from generalized triples (so they might be called {\it theta-vexillary}), but they are not the same as any of Billey-Lam's vexillary elements.

\medskip
\noindent
{\it Acknowledgements.}  We are grateful to S.~Billey for many helpful suggestions and comments.

\medskip
\noindent
{\it Notation.}  
We refer to \cite[\S8.1]{bb} for details on signed permutations, and review the notation and conventions here.  
We will consider permutations of (positive and negative) integers
\[
  \ldots, \bar{n}, \ldots, \bar{2},\bar{1},0,1,2,\ldots,n,\ldots,
\]
using the bar to denote a negative sign, and we take the natural order on them, as above.  All permutations are finite, in the sense that $v(m)=m$ whenever $|m|$ is sufficiently large.  We generally write permutations in $S_{2n+1}$ using one-line notation, by listing the values $v(\bar{n})\;v(\bar{n-1}) \cdots v(n)$.

A {\em signed permutation} is a permutation $w$ with the property that for each $i$, $w(\bar\imath) = \bar{w(i)}$.  A signed permutation is in $W_n$ if $w(m) = m$ for all $m>n$; this is a group isomorphic to the hyperoctahedral group, the Weyl group of types $B_n$ and $C_n$.  When writing signed permutations in one-line notation, we only list the values on positive integers: $w\in W_n$ is represented as $w(1)\;w(2)\;\cdots\;w(n)$.  For example, $w=\bar{2}\;1\;\bar{3}$ is a signed permutation in $W_3$, and $w(\bar{3}) = 3$ since $w(3)=\bar{3}$.  The {\it simple transpositions} $s_0,\ldots,s_n$ generate $W_n$, where for $i>0$, right-multiplication by $s_i$ exchanges entries in positions $i$ and $i+1$, and right-multiplication by $s_0$ replaces $w(1)$ with $\bar{w(1)}$.  The {\it length} of a signed permutation $w$ is the least number $\ell=\ell(w)$ such that $w = s_{i_1}\cdots s_{i_\ell}$. 
The \emph{longest element} in $W_n$, denoted $w_\circ^{(n)}$, is $\bar{1}\;\bar{2}\;\cdots\;\bar{n}$, and has length $n^2$.

The definition of $W_n$ presents it as embedded in the symmetric group $S_{2n+1}$, considering the latter as the group of all permutations of the integers $\bar{n},\ldots,0,\ldots,n$.  This is the \define{odd} embedding, and we write $\iota\colon W_n \hookrightarrow S_{2n+1}$ for emphasis when a signed permutation is considered as a full permutation.  Specifically, $\iota$ sends $w=w(1)\;w(2)\;\cdots\;w(n)$ to the permutation
\[
  \bar{w(n)}\;\cdots\;\bar{w(2)}\;\bar{w(1)}\;0\;w(1)\;w(2)\;\cdots\;w(n)
\]
in $S_{2n+1}$.  Occasionally, we will refer to the \define{even} embedding $\iota'\colon W_n \hookrightarrow S_{2n}$, defined similarly to $\iota$ by considering the even symmetric group as permutations of $\{\pm1,\ldots,\pm n\}$ and omitting the value $w(0)=0$.

A permutation $v$ has a \define{descent} at position $i$ if $v(i)>v(i+1)$; here $i$ may be any integer.  The same definition applies to signed permutations $w$, but we only consider descents at positions $i\geq 0$, following the convention of recording the values of $w$ only on positive integers.  A descent at $0$ simply means that $w(1)$ is negative.  For example, $w = \bar{2}\;1\;\bar{3}$ has descents at $0$ and $2$, while $\iota(w) = 3\;\bar{1}\;2\;0\;\bar{2}\;1\;\bar{3}$ has descents at $-3$, $-1$, $0$, and $2$.

\section{Diagrams and essential sets}\label{s.diagrams}

A permutation $v\in S_{2n+1}$ can be represented in a $(2n+1)\times(2n+1)$ array of boxes, with rows and columns indexed by $\{\bar{n},\ldots,0,\ldots,n\}$, by placing a dot in position $(v(i),i)$ for $\bar{n}\leq i\leq n$; we refer to this as the {\it permutation matrix} of $v$.  The {\it diagram} of $v$ is the collection of boxes that remain after crossing out those weakly south or east of a dot.  The {\it rank function} of $v$ is defined as
\begin{equation}\label{e.rank-typeA}
  r_v(p,q) = \#\{ i\leq \bar{p} \,|\, v(i) \geq q \},
\end{equation}
for $\bar{n} \leq p,q\leq n$.  This is equal to the number of dots strictly south and weakly west of the box $(q-1,\bar{p})$ in the permutation matrix of $v$.

The $(2n+1)^2$ numbers $r_v(p,q)$ determine $v$, but in fact much less information is required to specify a permutation.  A minimal list of rank conditions determining $v$ was given in \cite{fulton}, by restricting attention to the southeast corners of the diagram.\footnote{By analyzing the possibilities for boxes occurring as southeast corners of the diagram of a permutation, Eriksson and Linusson showed how to reconstruct $v$ from a subset of the essential set \cite{el}; however, the essential set is minimal in the sense of \cite[Lemma~3.10]{fulton}.}  More precisely, following the conventions and terminology of \cite{a-diag} (which differ slightly from those of \cite{fulton}), a pair $(p,q)$ is an {\it essential position} if the box $(q-1,\bar{p})$ is a southeast corner of the diagram of $v$.  The {\it essential set} $\Ess(v)$ is the set of $(k,p,q)$ such that $(p,q)$ is an essential position and $k=r_v(p,q)$.

In formulas, a box $(a,b)$ is a SE corner of the diagram of $v$ if and only if
\begin{align}
 v(b) &> a \geq v(b+1)  \quad \text{ and } \label{e.descents1} \\ 
 v^{-1}(a) &>b \geq v^{-1}(a+1) \tag{\theequation$'$}, \label{e.descents2}
\end{align}
so the essential positions are those $(p,q)$ such that $(q-1,\bar{p})$ satisfies Equations \eqref{e.descents1} and \eqref{e.descents2}.  This characterizes $\Ess(v)$ in terms of the descents of $v$, and will be useful later.

Analogous diagrams and essential sets for signed permutations were described in \cite{a-diag}.  We review the definitions and basic facts briefly here, since essential sets play a role in characterizing vexillary signed permutations.

For a signed permutation $w\in W_n$, the following simple lemma says that the essential positions of the corresponding permutation $\iota(w)\in S_{2n+1}$ are ``symmetric about the origin'' (see Figure~\ref{f.symmetric}).

\begin{lemma}[{\cite[Lemma~1.1]{a-diag}}]\label{l.symmetric-essential}
For $w\in W_n$, the essential set of $\iota(w)\in S_{2n+1}$ possesses the following symmetry: $(k,p,q)$ is in $\Ess(\iota(w))$ if and only if $(k+p+q-1,\bar{p}+1,\bar{q}+1)$ is in $\Ess(\iota(w))$.
\end{lemma}

\begin{figure}[ht]
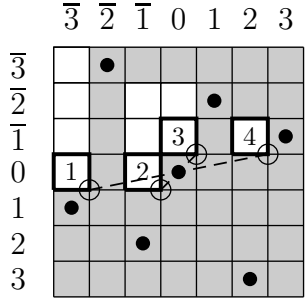


\pspicture(80,80)(-70,-50)

\psset{unit=1.35pt}

\pspolygon[fillstyle=solid,fillcolor=lightgray,linecolor=lightgray](-35,35)(35,35)(35,-35)(-35,-35)(-35,35)

\psline{-}(35,35)(-35,35)
\psline{-}(35,25)(-35,25)
\psline{-}(35,15)(-35,15)
\psline{-}(35,5)(-35,5)
\psline{-}(35,-5)(-35,-5)
\psline{-}(35,-15)(-35,-15)
\psline{-}(35,-25)(-35,-25)
\psline{-}(35,-35)(-35,-35)

\psline{-}(35,35)(35,-35)
\psline{-}(25,35)(25,-35)
\psline{-}(15,35)(15,-35)
\psline{-}(5,35)(5,-35)
\psline{-}(-5,35)(-5,-35)
\psline{-}(-15,35)(-15,-35)
\psline{-}(-25,35)(-25,-35)
\psline{-}(-35,35)(-35,-35)

\pscircle*(-30,-10){2}
\pscircle*(-20,30){2}
\pscircle*(-10,-20){2}
\pscircle*(0,0){2}
\pscircle*(10,20){2}
\pscircle*(20,-30){2}
\pscircle*(30,10){2}

\whitebox(-35,25)
\whitebox(-35,15)
\whitebox(-35,5)
\whitebox(-35,-5)
\whitebox(-15,15)
\whitebox(-15,5)
\whitebox(-15,-5)
\whitebox(-5,15)
\whitebox(-5,5)
\whitebox(15,5)

\psline[linewidth=1.5pt]{-}(-35,-5)(-35,5)(-25,5)(-25,-5)(-35,-5)

\psline[linewidth=1.5pt]{-}(-15,-5)(-15,5)(-5,5)(-5,-5)(-15,-5)

\psline[linewidth=1.5pt]{-}(-5,5)(-5,15)(5,15)(5,5)(-5,5)

\psline[linewidth=1.5pt]{-}(15,5)(15,15)(25,15)(25,5)(15,5)

\pscircle(-25,-5){3}
\pscircle(25,5){3}

\pscircle(-5,-5){3}
\pscircle(5,5){3}

\psline[linestyle=dashed,linewidth=0.5]{-}(-25,-5)(25,5)
\psline[linestyle=dashed,linewidth=0.5]{-}(-5,-5)(5,5)

\rput(-45,30){$\bar{3}$}
\rput(-45,20){$\bar{2}$}
\rput(-45,10){$\bar{1}$}
\rput(-45,0){$0$}
\rput(-45,-10){$1$}
\rput(-45,-20){$2$}
\rput(-45,-30){$3$}

\rput[b](-30,40){$\bar{3}$}
\rput[b](-20,40){$\bar{2}$}
\rput[b](-10,40){$\bar{1}$}
\rput[b](0,40){$0$}
\rput[b](10,40){$1$}
\rput[b](20,40){$2$}
\rput[b](30,40){$3$}

\rput(-30,0){\footnotesize{$1$}}
\rput(-10,0){\footnotesize{$2$}}

\rput(0,10){\footnotesize{$3$}}
\rput(20,10){\footnotesize{$4$}}

\endpspicture

\caption{Diagram and essential set for $v = \iota(\,\bar{2}\;3\;\bar{1}\,)$, with circled corners illustrating the symmetry of Lemma~\ref{l.symmetric-essential}. \label{f.symmetric}}
\end{figure}

\noindent
This implies that half of $\Ess(\iota(w))$ suffices to determine the signed permutation $w$; we will consider those corners appearing in the first $n$ columns.  In general, only a subset of these corners is needed, as shown in \cite{a-diag}.  

As with ordinary permutations, the essential set of a signed permutation is defined in terms of its diagram.  The \define{permutation matrix} of $w\in W_n$ is a $(2n+1)\times n$ array of boxes, with rows labelled $\bar{n},\ldots,0,\ldots,n$ and columns labelled $\bar{n},\ldots,\bar{1}$.  A dot is placed in the boxes $(\bar{w(i)},\bar{\imath})$ for $1\leq i\leq n$.  Additionally, for each dot, the box in the same column but opposite row is marked with an $\times$, as is each box to the right of this one.  (That is, an $\times$ is placed in each box $(a,b)$ such that $a=\bar{w(i)}$ for some $i\leq b$.)  The \define{extended diagram} of $w$ is the set $\tilde{D}_w$ of boxes that remain after crossing out those south or east of a dot in the permutation matrix; the \define{diagram} $D_w$ is the subset of $\tilde{D}_w$ not marked by an $\times$.  The placement of the $\times$'s is closely related to the parametrization of Schubert cells described in \cite{fp}, and the boxes of $D_w$ are in natural bijection with the inversions of $w$.  In particular, the number of boxes in $D_w$ is equal to the length of $w$.  See Figure~\ref{f.essential-ex1} for an illustration.

\begin{figure}[t]
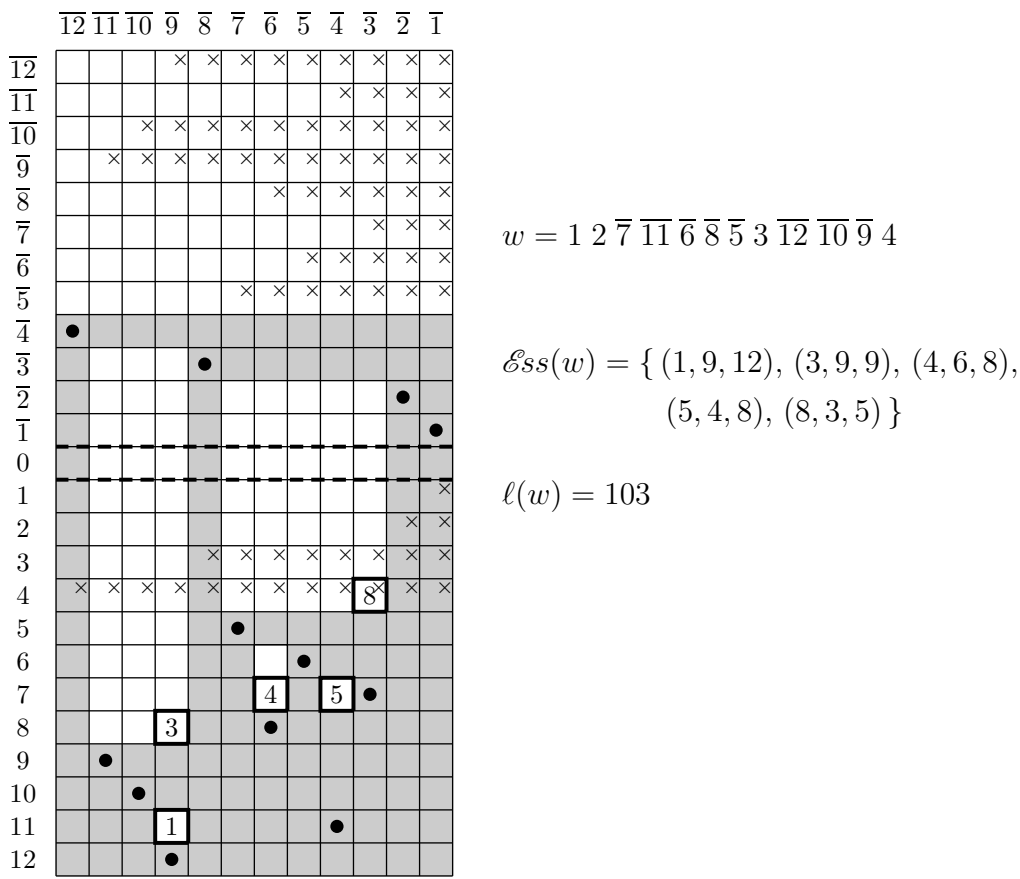

\pspicture(150,200)(-110,-160)

\psset{unit=1.25pt}

\pspolygon[fillstyle=solid,fillcolor=lightgray,linecolor=lightgray](-125,45)(-5,45)(-5,-125)(-125,-125)(-125,45)

\pspolygon[fillstyle=solid,fillcolor=white,linecolor=white](-115,35)(-85,35)(-85,-85)(-115,-85)(-115,35)

\pspolygon[fillstyle=solid,fillcolor=white,linecolor=white](-75,25)(-25,25)(-25,-45)(-75,-45)(-75,25)

\psline{-}(-5,125)(-125,125)
\psline{-}(-5,115)(-125,115)
\psline{-}(-5,105)(-125,105)
\psline{-}(-5,95)(-125,95)
\psline{-}(-5,85)(-125,85)
\psline{-}(-5,75)(-125,75)
\psline{-}(-5,65)(-125,65)
\psline{-}(-5,55)(-125,55)
\psline{-}(-5,45)(-125,45)
\psline{-}(-5,35)(-125,35)
\psline{-}(-5,25)(-125,25)
\psline{-}(-5,15)(-125,15)
\psline{-}(-5,5)(-125,5)
\psline{-}(-5,-5)(-125,-5)
\psline{-}(-5,-15)(-125,-15)
\psline{-}(-5,-25)(-125,-25)
\psline{-}(-5,-35)(-125,-35)
\psline{-}(-5,-45)(-125,-45)
\psline{-}(-5,-55)(-125,-55)
\psline{-}(-5,-65)(-125,-65)
\psline{-}(-5,-75)(-125,-75)
\psline{-}(-5,-85)(-125,-85)
\psline{-}(-5,-95)(-125,-95)
\psline{-}(-5,-105)(-125,-105)
\psline{-}(-5,-115)(-125,-115)
\psline{-}(-5,-125)(-125,-125)

\psline{-}(-5,125)(-5,-125)
\psline{-}(-15,125)(-15,-125)
\psline{-}(-25,125)(-25,-125)
\psline{-}(-35,125)(-35,-125)
\psline{-}(-45,125)(-45,-125)
\psline{-}(-55,125)(-55,-125)
\psline{-}(-65,125)(-65,-125)
\psline{-}(-75,125)(-75,-125)
\psline{-}(-85,125)(-85,-125)
\psline{-}(-95,125)(-95,-125)
\psline{-}(-105,125)(-105,-125)
\psline{-}(-115,125)(-115,-125)
\psline{-}(-125,125)(-125,-125)

{\footnotesize
\rput(-135,120){$\bar{12}$}
\rput(-135,110){$\bar{11}$}
\rput(-135,100){$\bar{10}$}
\rput(-135,90){$\bar{9}$}
\rput(-135,80){$\bar{8}$}
\rput(-135,70){$\bar{7}$}
\rput(-135,60){$\bar{6}$}
\rput(-135,50){$\bar{5}$}
\rput(-135,40){$\bar{4}$}
\rput(-135,30){$\bar{3}$}
\rput(-135,20){$\bar{2}$}
\rput(-135,10){$\bar{1}$}
\rput(-135,0){$0$}
\rput(-135,-10){$1$}
\rput(-135,-20){$2$}
\rput(-135,-30){$3$}
\rput(-135,-40){$4$}
\rput(-135,-50){$5$}
\rput(-135,-60){$6$}
\rput(-135,-70){$7$}
\rput(-135,-80){$8$}
\rput(-135,-90){$9$}
\rput(-135,-100){$10$}
\rput(-135,-110){$11$}
\rput(-135,-120){$12$}
}

{\footnotesize
\rput[b](-120,130){$\bar{12}$}
\rput[b](-110,130){$\bar{11}$}
\rput[b](-100,130){$\bar{10}$}
\rput[b](-90,130){$\bar{9}$}
\rput[b](-80,130){$\bar{8}$}
\rput[b](-70,130){$\bar{7}$}
\rput[b](-60,130){$\bar{6}$}
\rput[b](-50,130){$\bar{5}$}
\rput[b](-40,130){$\bar{4}$}
\rput[b](-30,130){$\bar{3}$}
\rput[b](-20,130){$\bar{2}$}
\rput[b](-10,130){$\bar{1}$}
}

\pscircle*(-120,40){2}
\pscircle*(-110,-90){2}
\pscircle*(-100,-100){2}
\pscircle*(-90,-120){2}
\pscircle*(-80,30){2}
\pscircle*(-70,-50){2}
\pscircle*(-60,-80){2}
\pscircle*(-50,-60){2}
\pscircle*(-40,-110){2}
\pscircle*(-30,-70){2}
\pscircle*(-20,20){2}
\pscircle*(-10,10){2}

\whitebox(-95,-115)

\whitebox(-65,-65)
\whitebox(-65,-75)

\whitebox(-45,-75)

\psline[linestyle=dashed,linewidth=1.5pt]{-}(-5,5)(-125,5)
\psline[linestyle=dashed,linewidth=1.5pt]{-}(-5,-5)(-125,-5)

\psline[linewidth=1.5pt]{-}(-85,-115)(-95,-115)(-95,-105)(-85,-105)(-85,-115)

\psline[linewidth=1.5pt]{-}(-85,-85)(-95,-85)(-95,-75)(-85,-75)(-85,-85)

\psline[linewidth=1.5pt]{-}(-55,-75)(-65,-75)(-65,-65)(-55,-65)(-55,-75)

\psline[linewidth=1.5pt]{-}(-35,-75)(-45,-75)(-45,-65)(-35,-65)(-35,-75)

\psline[linewidth=1.5pt]{-}(-25,-45)(-35,-45)(-35,-35)(-25,-35)(-25,-45)

\rput[bl](-120,-40){\tiny{$\times$}}
\rput[bl](-110,-40){\tiny{$\times$}}
\rput[bl](-100,-40){\tiny{$\times$}}
\rput[bl](-90,-40){\tiny{$\times$}}
\rput[bl](-80,-40){\tiny{$\times$}}
\rput[bl](-70,-40){\tiny{$\times$}}
\rput[bl](-60,-40){\tiny{$\times$}}
\rput[bl](-50,-40){\tiny{$\times$}}
\rput[bl](-40,-40){\tiny{$\times$}}
\rput[bl](-30,-40){\tiny{$\times$}}
\rput[bl](-20,-40){\tiny{$\times$}}
\rput[bl](-10,-40){\tiny{$\times$}}

\rput[bl](-110,90){\tiny{$\times$}}
\rput[bl](-100,90){\tiny{$\times$}}
\rput[bl](-90,90){\tiny{$\times$}}
\rput[bl](-80,90){\tiny{$\times$}}
\rput[bl](-70,90){\tiny{$\times$}}
\rput[bl](-60,90){\tiny{$\times$}}
\rput[bl](-50,90){\tiny{$\times$}}
\rput[bl](-40,90){\tiny{$\times$}}
\rput[bl](-30,90){\tiny{$\times$}}
\rput[bl](-20,90){\tiny{$\times$}}
\rput[bl](-10,90){\tiny{$\times$}}

\rput[bl](-100,100){\tiny{$\times$}}
\rput[bl](-90,100){\tiny{$\times$}}
\rput[bl](-80,100){\tiny{$\times$}}
\rput[bl](-70,100){\tiny{$\times$}}
\rput[bl](-60,100){\tiny{$\times$}}
\rput[bl](-50,100){\tiny{$\times$}}
\rput[bl](-40,100){\tiny{$\times$}}
\rput[bl](-30,100){\tiny{$\times$}}
\rput[bl](-20,100){\tiny{$\times$}}
\rput[bl](-10,100){\tiny{$\times$}}

\rput[bl](-90,120){\tiny{$\times$}}
\rput[bl](-80,120){\tiny{$\times$}}
\rput[bl](-70,120){\tiny{$\times$}}
\rput[bl](-60,120){\tiny{$\times$}}
\rput[bl](-50,120){\tiny{$\times$}}
\rput[bl](-40,120){\tiny{$\times$}}
\rput[bl](-30,120){\tiny{$\times$}}
\rput[bl](-20,120){\tiny{$\times$}}
\rput[bl](-10,120){\tiny{$\times$}}

\rput[bl](-80,-30){\tiny{$\times$}}
\rput[bl](-70,-30){\tiny{$\times$}}
\rput[bl](-60,-30){\tiny{$\times$}}
\rput[bl](-50,-30){\tiny{$\times$}}
\rput[bl](-40,-30){\tiny{$\times$}}
\rput[bl](-30,-30){\tiny{$\times$}}
\rput[bl](-20,-30){\tiny{$\times$}}
\rput[bl](-10,-30){\tiny{$\times$}}

\rput[bl](-70,50){\tiny{$\times$}}
\rput[bl](-60,50){\tiny{$\times$}}
\rput[bl](-50,50){\tiny{$\times$}}
\rput[bl](-40,50){\tiny{$\times$}}
\rput[bl](-30,50){\tiny{$\times$}}
\rput[bl](-20,50){\tiny{$\times$}}
\rput[bl](-10,50){\tiny{$\times$}}

\rput[bl](-60,80){\tiny{$\times$}}
\rput[bl](-50,80){\tiny{$\times$}}
\rput[bl](-40,80){\tiny{$\times$}}
\rput[bl](-30,80){\tiny{$\times$}}
\rput[bl](-20,80){\tiny{$\times$}}
\rput[bl](-10,80){\tiny{$\times$}}

\rput[bl](-50,60){\tiny{$\times$}}
\rput[bl](-40,60){\tiny{$\times$}}
\rput[bl](-30,60){\tiny{$\times$}}
\rput[bl](-20,60){\tiny{$\times$}}
\rput[bl](-10,60){\tiny{$\times$}}

\rput[bl](-40,110){\tiny{$\times$}}
\rput[bl](-30,110){\tiny{$\times$}}
\rput[bl](-20,110){\tiny{$\times$}}
\rput[bl](-10,110){\tiny{$\times$}}

\rput[bl](-30,70){\tiny{$\times$}}
\rput[bl](-20,70){\tiny{$\times$}}
\rput[bl](-10,70){\tiny{$\times$}}

\rput[bl](-20,-20){\tiny{$\times$}}
\rput[bl](-10,-20){\tiny{$\times$}}

\rput[bl](-10,-10){\tiny{$\times$}}

\rput(-90,-110){\footnotesize{$1$}}
\rput(-90,-80){\footnotesize{$3$}}
\rput(-60,-70){\footnotesize{$4$}}
\rput(-40,-70){\footnotesize{$5$}}
\rput(-30,-40){\footnotesize{$8$}}

\rput[l](10,70){$w=1\;2\;\bar{7}\;\bar{11}\;\bar{6}\;\bar{8}\;\bar{5}\; 3\;\bar{12}\;\bar{10}\;\bar{9}\; 4$}
\rput[l](10,30){$\Ess(w) = \{\, (1,9,12),\,(3,9,9),\,(4,6,8),$}
\rput[l](50,15){$\quad (5,4,8),\,(8,3,5) \,\}$}

\rput[l](10,-10){$\ell(w)=103$}



\endpspicture
\caption{Diagram and essential set of a signed permutation.  (This is also the vexillary signed permutation for the triple $\triple = ( \,1\;3\;4\;5\;8\, ,\, 9\;9\;6\;4\;3\, ,\, 12\;9\;8\;8\;5\,)$.)\label{f.essential-ex1}}
\end{figure}

The rank function of $w$ is defined as
\begin{equation}
  r_w(p,q) = \#\{ i \geq p \,|\, w(i) \leq \bar{q} \}.
\end{equation}
Since $w(\bar\imath) = \bar{w(i)}$, this is equivalent to
\begin{equation}
  r_w(p,q) = \#\{ i \leq \bar{p} \,|\, w(i) \geq q \}.
\end{equation}

The \define{essential set} of a signed permutation $w$ is the set of $(k,p,q)$ such that $(q-1,\bar{p})$ is a SE corner of the extended diagram $\tilde{D}_w$ and $k=r_w(p,q)$, with two exceptions.  First, if $p=1$ and $q<0$ (i.e., if the corner appears in the rightmost column and above the middle row in the permutation matrix), then $(k,p,q)$ is not in $\Ess(w)$.  Second, when $p>1$ and $q>0$, $(k,p,q)$ is not in $\Ess(w)$ if there is another SE corner in box $(\bar{q},\bar{p})$, and $k=r_w(p,q)=r_w(p,\bar{q}+1)-q+1$.

The exceptions are easy to understand in the context of linear algebra; see \cite{a-diag} for more explanation.  The first one comes from Lemma~\ref{l.symmetric-essential}: such corners are artifacts of restricting the permutation matrix to $n$ columns, and they do not appear as corners of the diagram of $\iota(w)$.  The second, more complicated exception never applies to the vexillary signed permutations to be defined in the next section.  As we will see in Lemma~\ref{l.ess-set}, when $w$ is vexillary, no SE corners of $\tilde{D}_w$ lie above the middle row, except possibly in the rightmost column.

As with ordinary permutations, a signed permutation is determined by its essential set: $w$ is the minimal element (in Bruhat order on $W_n$) such that $r_w(p,q)\geq k$ for all $(k,p,q)\in\Ess(w)$.  Furthermore, $\Ess(w)$ is the smallest set with this property---choosing any $(k_0,p_0,q_0)\in\Ess(w)$, there exists a $w'\neq w$ such that $r_{w'}(p_0,q_0)<k_0$ and $r_{w'}(p,q)\geq k$ for all $(k,p,q)\in\Ess(w)\setminus\{(k_0,p_0,q_0)\}$.  (This the content of \cite[Theorem~2.3]{a-diag}.)

\section{Triples and vexillary signed permutations}\label{s.triples}

As defined in the introduction, \define{triple} consists of three $s$-tuples of positive integers $\triple = (\bk,\bp,\bq)$, with
\begin{align*}
  \bk &= (0<k_1< \cdots < k_s), \\
  \bp &= (p_1 \geq \cdots \geq p_s >0 ), \\
  \bq &= (q_1 \geq \cdots \geq q_s >0 ),
\end{align*}
satisfying
\[
  p_i-p_{i+1} + q_i-q_{i+1} \geq k_{i+1}-k_i \tag{*} \label{e.triple}
\]
for $1\leq i\leq s-1$.  The triple is \define{essential} if the inequalities \eqref{e.triple} are strict.  
(These inequalities ensure that the rank conditions $\dim(E_{p_i}\cap F_{q_i})\geq k_i$ are feasible, and strict inequalities ensure they are independent.)  Each triple reduces to a unique essential triple, by successively removing each $(k_i,p_i,q_i)$ such that equality holds in \eqref{e.triple}.  Two triples are \define{equivalent} if they reduce to the same essential triple.

Given a triple, one forms a signed permutation $w(\triple)$ as follows.

\begin{enumerate}
\item[(1)] Starting in the $p_1$ position, place $k_1$ consecutive negative entries, in increasing order, ending with $\bar{q_1}$.  Mark these numbers as ``used''.

\smallskip

\item[(i)] For $1<i\leq s$, starting in the $p_i$ position, or the next available position to the right, fill the next available $k_i-k_{i-1}$ positions with negative entries chosen consecutively from the unused numbers, ending with at most $\bar{q_i}$.  (Note: the sequence may have to ``jump'' over previously placed sequences.)\footnote{It is useful to break each step in the construction into sub-steps.  In Step (i), first pick out the $k_i-k_{i-1}$ largest unused entries less than or equal to $\bar{q_i}$ and place them in a ``bin''.  Second, go to position $p_i$: If this position is available, take the smallest entry from the bin, place it here, and move to position $p_{i}+1$; if position $p_i$ is unavailable, just move to position $p_i+1$.  Repeat the second sub-step, starting at position $p_i+1$.  Finish when the bin is empty.}

\smallskip

\item[(s+1)] Fill the remaining available positions with the unused positive numbers, in increasing order.
\end{enumerate}

\begin{definition}
A signed permutation $w\in W_n$ is \define{vexillary} if $w=w(\triple)$ for some triple $\triple=(\bk,\bp,\bq)$.  By convention, the ``empty'' triple is a triple, so the identity element is vexillary.
\end{definition}

A triple also defines a strict partition $\lambda(\triple)$, by setting $\lambda_{k_i} = p_i+q_i-1$, and defining the other parts minimally so that the result is a strictly decreasing sequence of integers.  (In formulas, $\lambda_k = p_i+q_i-1+k_i-k$ if $k_{i-1}< k \leq k_i$.)

For example, with $\triple = ( \,1\;3\;4\;5\;8\, ,\, 9\;9\;6\;4\;3\, ,\, 12\;9\;8\;8\;5\,)$, the six steps in forming $w(\triple)$ produce
\begin{align*}
  w &= \cdot\;\cdot\;\cdot\;\cdot\;\cdot\;\cdot\;\cdot\;\cdot\;{\bf\bar{12}}, \\
  w &= \cdot\;\cdot\;\cdot\;\cdot\;\cdot\;\cdot\;\cdot\;\cdot\;\bar{12}\;{\bf\bar{10}\;\bar{9}}, \\
  w &= \cdot\;\cdot\;\cdot\;\cdot\;\cdot\;{\bf\bar{8}}\;\cdot\;\cdot\;\bar{12}\;\bar{10}\;\bar{9}, \\
  w &= \cdot\;\cdot\;\cdot\;{\bf\bar{11}}\;\cdot\;\bar{8}\;\cdot\;\cdot\;\bar{12}\;\bar{10}\;\bar{9}, \\
  w &= \cdot\;\cdot\;{\bf\bar{7}}\;\bar{11}\;{\bf\bar{6}}\;\bar{8}\;{\bf\bar{5}}\;\cdot\;\bar{12}\;\bar{10}\;\bar{9}, \\
  w &= {\bf 1}\;{\bf 2}\;\bar{7}\;\bar{11}\;\bar{6}\;\bar{8}\;\bar{5}\;{\bf 3}\;\bar{12}\;\bar{10}\;\bar{9}\;{\bf 4}.
\end{align*}
The corresponding partition is $\lambda = (20, 18, 17, 13, 11, 9, 8, 7)$.

In the rest of the article, we study some of the properties of this construction.  Equivalent triples produce the same signed permutation and the same strict permutation, so we will generally assume triples are essential.

\begin{lemma}\label{l.descents}
Let $w=w(\triple)$, for an essential triple $\triple=(\bk,\bp,\bq)$.  The descents of $w$ are at the positions $p_i-1$.  In fact, for each $i$, we have $w(p_i-1)>\bar{q_i}\geq w(p_i)$, and there are no other descents.
\end{lemma}

\begin{proof}
In Step (1), no descents are created, unless $p_1=1$, in which case the permutation has a single descent at $0$.

Now for $1<i\leq s$, consider the situation before Step (i) in constructing $w(\triple)$.  Assume inductively that for $j<i$, there is a descent at position $p_j-1$ whenever this position has been filled, satisfying $w(p_i-1)>\bar{q_i}\geq w(p_i)$, and there are no other descents.

In carrying out Step (i), we place negative entries in consecutive vacant positions, from left to right, starting at position $p_i$ (or the next vacant position to the right, if $p_i=p_{i-1}$).  We consider ``sub-steps'' of Step (i), where we are placing an entry at position $p\geq p_i$, and distinguish three cases.

First, suppose we are at position $p$, with $p<p_{i-1}-1$.  In this case, the previous entry placed in Step (i) (if any) was placed at position $p-1$, so we did not create a descent at $p-1$.  Position $p+1$ is still vacant, so no new descents are created.

To illustrate, take $\triple = ( \,1\;3\;4\;5\;8\, ,\, 9\;9\;6\;4\;3\, ,\, 12\;9\;8\;8\;5\,)$, as in the example above.  In Step (3), we place $\bar{8}$ without creating a descent:
\[
  w = \cdot\;\cdot\;\cdot\;\cdot\;\cdot\;{\bf\bar{8}}\;\cdot\;\cdot\;\bar{12}\;\bar{10}\;\bar{9}.
\]

Next, suppose we are at position $p=p_{i-1}-1$.  This means $p_{i-1}-p_{i}\leq k_{i}-k_{i-1}$, so let $\beta = (k_{i}-k_{i-1})-(p_{i-1}-p_i)$ be the number of entries remaining to be placed in Step (i), after placing the current one.  The inequality \eqref{e.triple} in the triple condition implies $q_i+\beta<q_{i-1}$, which in turn means that the entries used in previous steps are all strictly less than $\bar{q_i+\beta}$; therefore the $\beta+1$ entries $\bar{q_i+\beta},\bar{q_i+\beta-1},\ldots,\bar{q_i}$ are all available to be placed.  It follows that $w(p_{i-1}-1) = \bar{q_i+\beta}$, which creates a descent satisfying the claim of the lemma.  Again, we did not create a descent at position $p-1$, for the same reason as in the previous case.

Continuing the running example, the first entry placed in Step (5) is the $\bar{7}$, creating a descent at position 3:
\[
  w = \cdot\;\cdot\;{\bf\bar{7}}\;\bar{11}\;{\cdot}\;\bar{8}\;{\cdot}\;\cdot\;\bar{12}\;\bar{10}\;\bar{9}.
\]
(In this case, we had $q_5=5$ and $\beta = 2$.)

Finally, suppose we are at position $p=p_{i-1}$.  Set $\beta = (k_i-k_{i-1})-(p_{i-1}-p_i) -1$, so $\beta \geq 0$; this is the number of entries to be placed after the current one.  Now the essential triple condition implies $q_i+ \beta + 1 < q_{i-1}$.  The entries that will be placed are therefore $-q_i-\beta,-q_i-\beta+1,\ldots,-q_i$.  These are all greater than the ones that already occupy positions to the right of $p_{i-1}$.  Whenever such an entry is placed in a vacant position to the right of a filled position, then, it does not create a descent at that filled position.  Whenever it is placed to the left of a filled position, it does create a descent at the position it is placed.  The latter can only happen at positions $p_j-1$, for $j<i-1$.

In Step (5) of our example, after placing the $\bar{7}$ in position $p_5=p_4-1$, we come to position $p_4$.  The entries to be placed are $\bar{6},\bar{5}$.  In placing the $\bar{6}$ in the next vacant position---position $5$---we do not create a descent at the filled position to its left, but we do create a descent at position $5$, since position $6$ is already filled:
\[
  w = \cdot\;\cdot\;{\bar{7}}\;\bar{11}\;{\bf\bar{6}}\;\bar{8}\;{\bf\bar{5}}\;\cdot\;\bar{12}\;\bar{10}\;\bar{9}.
\]
In placing the $\bar{5}$, there is no descent created, since the position to its right is vacant.

At Step (s+1), the same reasoning as in the last case considered shows that descents are created at those $p_j-1$ which are still vacant, and nowhere else.
\end{proof}

Given a triple $\triple=(\bk,\bp,\bq)$, the \define{dual triple} is $\triple^*=(\bk,\bq,\bp)$.

\begin{lemma}\label{l.inverse}
We have $w(\triple^*)=w(\triple)^{-1}$.
\end{lemma}

\begin{proof}
Suppose $w\in W_n$, and consider $\iota(w)$ as a bijective map from the set $\{\bar{n},\ldots,0,\ldots,n\}$ to itself.  To determine $\iota(w)$, it is enough to know its values on any set of integers whose absolute values are $\{1,\ldots,n\}$.  The construction of $w$ translates as follows.  Let $a(1)=\{p_1,p_1+1,\ldots,p_1+k_1-1\}$, let $a(2)$ be the set of $k_2-k_1$ consecutive integers in $\{1,\ldots,n\}\setminus a(1)$ starting from the $p_2$th element, and define sets $a(3),\ldots,a(s)$ similarly.  Let $a(s+1)$ be the complement of $a(1),\ldots,a(s)$, so the sets $a(1),\ldots,a(s+1)$ partition $\{1,\ldots,n\}$.  Define $b(1),\ldots,b(s+1)$ in the same way, using $\bq$ in place of $\bp$.  Also write $\bar{a(i)}$ and $\bar{b(i)}$ for the corresponding sets of negative integers.  Now $\iota(w)$ maps $a(i)$ to $\bar{b(i)}$ for $1\leq i\leq s$, and it maps $a(s+1)$ to $b(s+1)$.  It follows that $\iota(w)^{-1}$ maps $\bar{b(i)}$ to $a(i)$, or equivalently, $b(i)$ to $\bar{a(i)}$, and maps $b(s+1)$ to $a(s+1)$.  In other words, the inverse is obtained by switching the roles of $a$ and $b$, so the lemma is proved.
\end{proof}

\begin{lemma}\label{l.ess-set}
{\rm(1)}\; Let $w=w(\triple)$ be a vexillary signed permutation, for an essential triple $\triple=(\bk,\bp,\bq)$.  Then the SE corners of the diagram of $\iota(w)$ are the boxes $(q_i-1,\, \bar{p_i})$, together with their reflections $(\bar{q_i},\, p_i-1)$.  (In particular, no box $(a,b)$ with $a,b < 0$ occurs as a SE corner.)

\medskip
\noindent
{\rm (2)}\; $k_i$ is the number of dots strictly south and weakly west of the $i$th SE corner in the diagram (in position $(q_i-1,\, \bar{p_i})$).  
\end{lemma}

\begin{proof}
By Lemma~\ref{l.symmetric-essential}, it suffices to show that the first $s$ corners of the diagram of $\iota(w)$ (ordered SW to NE) are as claimed.  For $p>0$, a signed permutation $w$ has a descent at position $p-1$ iff $\iota(w)$ has descents at positions $p-1$ and $\bar{p}$.  By Lemma~\ref{l.descents}, the descents of $\iota(w)$ are at positions $p_i-1$ and $\bar{p_i}$, and the inequalities \eqref{e.descents1}
\[
  \iota(w)(\bar{p_i}) > q_i-1 \geq \iota(w)(\bar{p_i-1})
\]
follow from
\[
 w(p_i-1) \geq q_i-1 > w(p_i),
\]
using Lemma~\ref{l.descents} again.  

A similar argument establishes the inequalities \eqref{e.descents2}:
\[
  \iota(w)^{-1}(q_i-1) > \bar{p_i} \geq \iota(w)^{-1}(q_i)
\]
(Swap $\bp$ and $\bq$, and apply Lemma~\ref{l.descents} to $w^{-1}=w(\bk,\bq,\bp)$.)

Part (2) is easy from the construction: At the end of Step (i), all the entries to the right of $p_i$ are at most $\bar{q_i}$, and there are $k_i$ of them.  After Step (i), any entry placed to the right of $p_i$ is greater than $\bar{q_i}$.  
\end{proof}

\begin{theorem}\label{t.vex-characterization}
For $w\in W_n$, the following are equivalent:
\begin{enumerate}
\item The signed permutation $w$ is vexillary. \label{cond1}

\smallskip

\item The essential positions $(p_i,q_i)$ of $w$ can be ordered so that $p_1 \geq \cdots \geq p_s >0$ and $q_1 \geq \cdots \geq q_s >0$. \label{cond2}

\smallskip

\item The permutation $\iota(w)$ is vexillary (as an element of $S_{2n+1}$). \label{cond3}
\end{enumerate}
\end{theorem}

\begin{proof}
We will show \eqref{cond1} $\Rightarrow$ \eqref{cond2} $\Rightarrow$ \eqref{cond3} $\Rightarrow$ \eqref{cond1}.  The implication \eqref{cond1} $\Rightarrow$ \eqref{cond2} is immediate from Lemma~\ref{l.ess-set}, and \eqref{cond2} $\Rightarrow$ \eqref{cond3} follows from Lemma~\ref{l.symmetric-essential}.  It remains to show \eqref{cond3} $\Rightarrow$ \eqref{cond1}.

Suppose $\iota(w)$ is vexillary, and recall that this is equivalent to requiring that the SE corners of its diagram proceed from southwest to northeast \cite[Remark~9.17]{fulton}.  Take those SE corners $(a,b)$ such that $b<0$; Lemma~\ref{l.symmetric-essential} implies $a\geq 0$.  Let $s=\lceil \frac{\#\Ess(\iota(w))}{2} \rceil$ be the number of such boxes.  Reading from SW to NE, label them $(a_i,b_i)$, with $1\leq i\leq s$.  Set $p_i=-b_i$ and $q_i=a_i+1$, and let $k_i$ be the number of dots strictly south and weakly west of $(a_i,b_i)$.

We claim that $\triple=(\bk,\bp,\bq)$ is an essential triple.  
Indeed, looking at the extended diagram $\tilde{D}_w$, the fact that the boxes $(q_i-1,\bar{p_i})$ are SE corners implies the inequalities \eqref{e.triple}.  
It follows that $w$ is equal to $w(\triple)$, since a permutation is determined by its essential set \cite[Lemma~3.10(b)]{fulton}, \cite[Theorem~2.3]{a-diag}.  
\end{proof}

\section{Pattern avoidance}\label{s.patterns}

Given a permutation $\pi$ in $S_m$, a permutation $v$ {\em contains the pattern $\pi$} if, when written in one-line notation, there is an $m$-element subsequence of $v$ which is in the same relative order as $\pi$; otherwise $v$ {\em avoids} $\pi$.  There is a similar notion for signed permutations and signed patterns.  Given a signed pattern $\pi = \pi(1)\;\pi(2)\cdots\pi(m)$ in $W_m$, a signed permutation $w$ {\em contains} $\pi$ if there is a subsequence $w(i_1)\cdots w(i_m)$ such that the signs of $w(i_j)$ and $\pi(j)$ are the same for all $j$, and also the absolute values of the subsequence are in the same relative order as the absolute values of $\pi$.  Otherwise $w$ {\em avoids} $\pi$.  (See, e.g., \cite[Definition 6]{bl}.)

For example, $\bar{5}\;1\;3\;\bar{2}\;4$ contains the pattern $[\bar{3}\;2\;\bar{1}]$, as the subsequence $\bar{5}\;3\;\bar{2}$, but $\bar{5}\;1\;2\;\bar{3}\;\bar{4}$ avoids $[\bar{3}\;2\;\bar{1}]$.

\begin{proposition}\label{p.patterns}
A signed permutation $w$ is vexillary if and only if $w$ avoids the signed patterns $[2\;1]$, $[\bar{3}\;2\;\bar{1}]$, $[\bar{4}\;\bar{1}\;\bar{2}\;3]$, $[\bar{4}\;1\;\bar{2}\;3]$, $[\bar{3}\;\bar{4}\;\bar{1}\;\bar{2}]$, $[\bar{3}\;\bar{4}\;1\;\bar{2}]$, $[\bar{2}\;\bar{3}\;4\;\bar{1}]$, $[2\;\bar{3}\;4\;\bar{1}]$, and $[3\;\bar{4}\;\bar{1}\;\bar{2}]$. \end{proposition}

\begin{proof}
We will use the characterization of Theorem~\ref{t.vex-characterization}, and show that a signed permutation $w$ avoids these nine patterns if and only if $\iota(w)$ is vexillary in $S_{2n+1}$.  Recall that a permutation is vexillary if and only if it avoids the pattern $[2\;1\;4\;3]$.  When embedded by $\iota$, the nine signed patterns listed in \eqref{cond2} all contain $[2\;1\;4\;3]$, so if $\iota(w)$ avoids $[2\;1\;4\;3]$, then $w$ must avoid these signed patterns; this proves the ``$\Rightarrow$'' direction.

For the ``$\Leftarrow$'' implication, we claim that if $\iota(w)$ contains the pattern $[2\;1\;4\;3]$, then $w$ contains one of the nine signed patterns listed in the Proposition.  To check this, we break the ways $[2\;1\;4\;3]$ can appear in $\iota(w)$ into cases, and observe that in each case, one of the listed patterns appears.

For keeping track of the cases, we introduce some temporary notation.  An instance of the pattern $[2\;1\;4\;3]$ is witnessed by $\iota(w)(b) < \iota(w)(a) < \iota(w)(d) < \iota(w)(c)$ for some $a<b<c<d$.  We record which of $a,b,c,d$ are negative by placing a vertical bar in the word $[2\;1\;4\;3]$, and which of the corresponding values of $\iota(w)$ are negative by placing underlines.  For example, if $w=2\;\bar{3}\;\bar{5}\;4\;\bar{1}$, so
\[
  \iota(w) = 1\;\bar{4}\;5\;3\;\bar{2}\;0\;2\;\bar{3}\;\bar{5}\;4\;\bar{1}, 
\]
then an instance of $[2\;1\;4\;3]$ occurs as $[\ul{2}\,|\,\ul{1}\;4\;\ul{3}]$ in the subsequence $\bar{4}\;\bar{5}\;4\;\bar{1}$ (among others), and another occurs as $[|\,\ul{2}\;\ul{1}\;4\;\ul{3}]$ in the subsequence $\bar{3}\;\bar{5}\;4\;\bar{1}$.

We now check all cases, organized by position of the vertical bar and underlines.

Any occurrence of $[|\,2\;1\;4\;3]$, $[|\,2\;\ul{1}\;4\;3]$, or $[|\,\ul{2}\;\ul{1}\;4\;3]$ contains the signed pattern $[2\;1]$.  Suppose $\alpha<\beta<\gamma$, so an occurrence of $[|\,\ul{2}\;\ul{1}\;4\;\ul{3}]$ is a subsequence $\bar{\beta}\;\bar{\gamma}\;x\;\bar{\alpha}$.  This is the signed pattern $[\bar{2}\;\bar{3}\;4\;\bar{1}]$ if $x>\gamma$, it is $[\bar{3}\;\bar{4}\;1\;\bar{2}]$ if $x<\alpha$, and it contains $[\bar{3}\;2\;\bar{1}]$ if $\alpha<x<\gamma$.  Finally, an occurrence of $[|\,\ul{2}\;\ul{1}\;\ul{4}\;\ul{3}]$ is the signed pattern $[\bar{3}\;\bar{4}\;\bar{1}\;\bar{2}]$.

The cases $[2\,|\,1\;4\;3]$, $[2\,|\,\ul{1}\;4\;3]$, and $[\ul{2}\,|\,\ul{1}\;4\;3]$ all contain the signed pattern $[2\;1]$, as before.  With $a<b<c$, the case $[\ul{2}\,|\,\ul{1}\;4\;\ul{3}]$ depends on the position of the ``$\ul{2}$''.  A subsequence $\beta\;\bar{\gamma}\;x\;\bar{\alpha}$ is the signed pattern $[2\;\bar{3}\;4\;\bar{1}]$ if $x>\gamma$; it is $[3\;\bar{4}\;1\;\bar{2}]$ if $x<\alpha$, and it contains $[\bar{3}\;2\;\bar{1}]$ if $\alpha<x<\gamma$.  The subsequences $\bar{\gamma}\;\beta\;x\;\bar{\alpha}$ and $\bar{\gamma}\;x\;\beta\;\bar{\alpha}$ contain $[\bar{3}\;2\;\bar{1}]$.  The subsequence $\bar{\gamma}\;x\;\bar{\alpha}\;\beta$ is the signed pattern $[\bar{4}\;1\;\bar{2}\;3]$ if $x<\alpha$, it contains $[\bar{3}\;2\;\bar{1}]$ if $\alpha<x<\beta$, and it contains $[2\;1]$ if $x>\beta$.  Similarly, depending on the position of the ``$\ul{2}$'',  the case $[\ul{2}\,|\,\ul{1}\;\ul{4}\;\ul{3}]$ contains $[3\;\bar{4}\;\bar{1}\;\bar{2}]$, $[\bar{3}\;2\;\bar{1}]$, or $[\bar{4}\;\bar{1}\;\bar{2}\;3]$.

All five cases $[2\;1\,|\,4\;3]$, $[2\;\ul{1}\,|\,4\;3]$, $[\ul{2}\;\ul{1}\,|\,4\;3]$, $[\ul{2}\;\ul{1}\,|\,4\;\ul{3}]$, and $[\ul{2}\;\ul{1}\,|\,\ul{4}\;\ul{3}]$ contain the signed pattern $[2\;1]$.

The remaining cases are symmetric to the ones checked above.
\end{proof}

\begin{remark}\label{r.even}
An identical analysis shows that $w$ avoids the same nine signed patterns if and only if $\iota'(w) \in S_{2n}$ avoids $[2\;1\;4\;3]$; therefore $w$ is vexillary if and only if $\iota'(w)$ is vexillary.
\end{remark}

The pattern avoidance criterion for ordinary permutations allowed J.~West to enumerate the vexillary permutations in $S_n$ \cite{west}.  His proof uses a bijection between permutations avoiding $[2\;1\;4\;3]$ (vexillary permutations) and those that avoid $[4\;3\;2\;1]$.  One might expect a similar enumeration of vexillary signed permutations.  In fact, for the even embedding $\iota'\colon W_n \hookrightarrow S_{2n}$ (which omits ``$w(0)=0$''), Egge showed that the number of $w\in W_n$ such that $\iota'(w)$ avoids $[4\;3\;2\;1]$ is equal to
\[
  V_n := \sum_{k=0}^n \binom{n}{k}^2 C_k,
\]
where $C_k = \frac{1}{k+1}\binom{2k}{k}$ is the $k$th Catalan number \cite{egge}.  Together with computer verification up to $n=7$, this suggests the conjecture that $V_n$ is also the number of vexillary signed permutations in $W_n$.  West's bijection does not preserve the subgroup of signed permutations, though, so the proof is not immediately clear.

\section{Labelled Young diagrams}\label{s.lyd}

There is an alternative way to encode vexillary signed permutations, useful for determining when right-multiplication by a simple reflection takes one vexillary signed permutation to another.  Identifying a strict partition $\lambda$ with its shifted Young diagram, a \define{labelled Young diagram} of shape $\lambda$ is an assignment of integers to the SE corners, weakly decreasing from top to bottom, such that the labels $m_1,\ldots,m_s$ in rows $k_1,\ldots,k_s$ satisfy $0\leq m_i<\lambda_{k_i}$ and $m_i-m_{i+1}\leq \lambda_{k_i}-\lambda_{k_{i+1}}$.  (Note that $\lambda_{k_i}-\lambda_{k_{i+1}}+1$ is the number of boxes in the rim-hook connecting the corners labelled $m_i$ and $m_{i+1}$.)

Labelled Young diagrams are in bijection with essential triples: given $\triple=(\bk,\bp,\bq)$, the shifted Young diagram for $\lambda(\triple)$ has SE corners in rows $k_1,\ldots,k_s$, and one forms a labelled Young diagram of this shape by placing integers $m_i=p_i-1$ in these corners.  Examples are shown in Figure~\ref{f.lyd}.

\begin{figure}
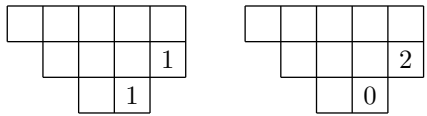


\pspicture(70,40)(-100,-20)

\psset{unit=1.35pt}

\psline{-}(-50,20)(0,20)
\psline{-}(-50,10)(0,10)
\psline{-}(-40,0)(0,0)
\psline{-}(-30,-10)(-10,-10)

\psline{-}(-50,20)(-50,10)
\psline{-}(-40,20)(-40,0)
\psline{-}(-30,20)(-30,-10)
\psline{-}(-20,20)(-20,-10)
\psline{-}(-10,20)(-10,-10)
\psline{-}(0,20)(0,0)

\rput(-5,5){\footnotesize{$1$}}
\rput(-15,-5){\footnotesize{$1$}}

\endpspicture
\pspicture(70,40)(-20,-20)

\psset{unit=1.35pt}

\psline{-}(-50,20)(0,20)
\psline{-}(-50,10)(0,10)
\psline{-}(-40,0)(0,0)
\psline{-}(-30,-10)(-10,-10)

\psline{-}(-50,20)(-50,10)
\psline{-}(-40,20)(-40,0)
\psline{-}(-30,20)(-30,-10)
\psline{-}(-20,20)(-20,-10)
\psline{-}(-10,20)(-10,-10)
\psline{-}(0,20)(0,0)

\rput(-5,5){\footnotesize{$2$}}
\rput(-15,-5){\footnotesize{$0$}}

\endpspicture

\caption{Labelled Young diagrams associated to $\triple = (\;2\;3,\;2\;2,\;3\;1\;)$ and $\triple^* = (\;2\;3,\;3\;1,\;2\;2\;)$.  The corresponding vexillary signed permutations are $2\;\bar{4}\;\bar{3}\;\bar{1}$ and $\bar{4}\;1\;\bar{3}\;\bar{2}$, respectively.\label{f.lyd}}
\end{figure}

Let $w=w(\triple)$, and let $Y$ be the corresponding labelled Young diagram of shape $\lambda(\triple)$.  By Lemma~\ref{l.descents}, the descents of $w$ are the corner labels of $Y$, i.e., $\ell(w s_m) = \ell(w)-1$ if and only if $m$ is a label.

A corner label $m$ in $Y$ is \define{removable} if
\begin{itemize}
\item it appears in $Y$ exactly once, and

\smallskip

\item when $m=0$, the row in which it appears contains a single box.

\end{itemize}
Given a removable label $m$, one can remove the box containing it to form a new labelled Young diagram $Y\setminus m$, whose corners are labelled according to the following four rules.
\begin{enumerate}
\item If a corner of $Y\setminus m$ is also a corner of $Y$, its label is the same.

\smallskip

\item If removing $m$ produces a new corner one box to the left, label that corner $m-1$.

\smallskip

\item If removing $m$ produces a new corner one box above, label that corner $m+1$.

\smallskip

\item If removing $m$ produces two new corners, label the one to the left $m-1$ and the one above $m+1$.
\end{enumerate}
These are the only possibilities for removing a corner from the Young diagram.  Examples are shown in Figure~\ref{f.removal}.

\begin{figure}
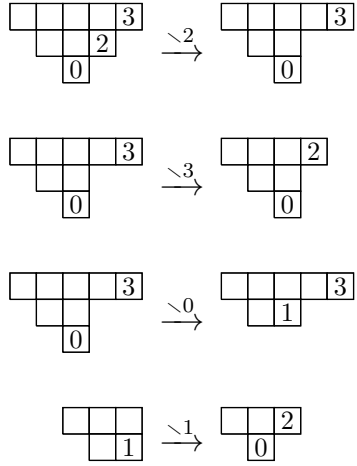


\pspicture(60,40)(-100,-10)

\psset{unit=1.00pt}

\whitebox(-50,10)
\whitebox(-40,10)
\whitebox(-30,10)
\whitebox(-20,10)
\whitebox(-10,10)

\whitebox(-40,0)
\whitebox(-30,0)
\whitebox(-20,0)

\whitebox(-30,-10)

\rput(-5,15){\footnotesize{$3$}}
\rput(-15,5){\footnotesize{$2$}}
\rput(-25,-5){\footnotesize{$0$}}

\rput(15,5){$\xrightarrow{\setminus 2}$}

\endpspicture
\pspicture(60,40)(-20,-10)

\psset{unit=1.00pt}

\whitebox(-50,10)
\whitebox(-40,10)
\whitebox(-30,10)
\whitebox(-20,10)
\whitebox(-10,10)

\whitebox(-40,0)
\whitebox(-30,0)

\whitebox(-30,-10)

\rput(-5,15){\footnotesize{$3$}}
\rput(-25,-5){\footnotesize{$0$}}

\endpspicture

\pspicture(60,40)(-100,-10)

\psset{unit=1.00pt}

\whitebox(-50,10)
\whitebox(-40,10)
\whitebox(-30,10)
\whitebox(-20,10)
\whitebox(-10,10)

\whitebox(-40,0)
\whitebox(-30,0)

\whitebox(-30,-10)

\rput(-5,15){\footnotesize{$3$}}
\rput(-25,-5){\footnotesize{$0$}}

\rput(15,5){$\xrightarrow{\setminus 3}$}

\endpspicture
\pspicture(60,40)(-20,-10)

\psset{unit=1.00pt}

\whitebox(-50,10)
\whitebox(-40,10)
\whitebox(-30,10)
\whitebox(-20,10)

\whitebox(-40,0)
\whitebox(-30,0)

\whitebox(-30,-10)

\rput(-15,15){\footnotesize{$2$}}
\rput(-25,-5){\footnotesize{$0$}}

\endpspicture

\pspicture(60,40)(-100,-10)

\psset{unit=1.00pt}

\whitebox(-50,10)
\whitebox(-40,10)
\whitebox(-30,10)
\whitebox(-20,10)
\whitebox(-10,10)

\whitebox(-40,0)
\whitebox(-30,0)

\whitebox(-30,-10)

\rput(-5,15){\footnotesize{$3$}}
\rput(-25,-5){\footnotesize{$0$}}

\rput(15,5){$\xrightarrow{\setminus 0}$}

\endpspicture
\pspicture(60,40)(-20,-10)

\psset{unit=1.00pt}

\whitebox(-50,10)
\whitebox(-40,10)
\whitebox(-30,10)
\whitebox(-20,10)
\whitebox(-10,10)

\whitebox(-40,0)
\whitebox(-30,0)


\rput(-5,15){\footnotesize{$3$}}
\rput(-25,5){\footnotesize{$1$}}

\endpspicture

\pspicture(60,40)(-100,0)

\psset{unit=1.00pt}

\whitebox(-30,10)
\whitebox(-20,10)
\whitebox(-10,10)

\whitebox(-20,0)
\whitebox(-10,0)


\rput(-5,5){\footnotesize{$1$}}

\rput(15,10){$\xrightarrow{\setminus 1}$}

\endpspicture
\pspicture(60,40)(-20,0)

\psset{unit=1.00pt}

\whitebox(-50,10)
\whitebox(-40,10)
\whitebox(-30,10)

\whitebox(-40,0)


\rput(-25,15){\footnotesize{$2$}}
\rput(-35,5){\footnotesize{$0$}}

\endpspicture

\caption{Removing corner labels. \label{f.removal}}
\end{figure}

If $w$ is the vexillary signed permutation corresponding to a labelled Young diagram $Y$, then $ws_m$ is vexillary when $m$ is removable: the reader may verify that $Y\setminus m$ is its corresponding labelled Young diagram.  The next theorem says the converse also holds.

\begin{theorem}
Let $Y$ be the labelled Young diagram corresponding to a vexillary signed permutation $w$.  Then $ws_m$ is vexillary of length $\ell(w)-1$ if and only if $m$ is a removable label of $Y$.
\end{theorem}

\begin{proof}
Let $\triple=(\bk,\bp,\bq)$ be the corresponding essential triple.  Using the pattern avoidance criterion for vexillarity, we will show that if a label $m$ is not removable, then $ws_m$ is not vexillary.

First suppose $m$ occurs more than once as a label in $Y$, or equivalently, the index $p=m+1$ is repeated in the sequence $\bp$.  From the construction of $w(\triple)$, the repetition of $p$ means there is a gap in the increasing sequence of negative integers starting at position $p$ in $w$.  In other words, there is a ``partial pattern'' $[\bar{3}\;\bar{1}]$ in $w$, with the $\bar{3}$ occurring at position $p$.

If $m=0$, then the gap must be filled by a positive integer occurring to the right.  That is, the pattern $[\bar{3}\;\bar{1}\;2]$ occurs in $w$, with the ``$\bar{3}$ in position $1$.  We find that $ws_0$ contains the pattern $[3\;\bar{1}\;2]$, so in particular it contains $[2\;1]$ and is not vexillary.

If $m>0$, and the entry of $w$ in position $m=p-1$ is positive, then one of the following patterns appears:
\begin{align*}
 [1\;\bar{4}\;\bar{2}\;3] & \text{ (with ``}1\text{'' at position }p-1\text{ and ``}\bar{4}\text{'' at position }p), \\
 [\bar{3}\;1\;\bar{4}\;\bar{2}] & \text{ (with ``}1\text{'' at position }p-1\text{ and ``}\bar{4}\text{'' at position }p), \\
 [2\;\bar{3}\;\bar{1}] & \text{ (with ``}2\text{'' at position }p-1\text{ and ``}\bar{3}\text{'' at position }p), \\
 [2\;4\;\bar{3}\;\bar{1}] & \text{ (with ``}4\text{'' at position }p-1\text{ and ``}\bar{3}\text{'' at position }p),\text{ or} \\
 [\bar{2}\;4\;\bar{3}\;\bar{1}] & \text{ (with ``}4\text{'' at position }p-1\text{ and ``}\bar{3}\text{'' at position }p).
\end{align*}
If the entry in position $p-1$ is negative, it must form a partial pattern $[\bar{1}\;\bar{4}\;\bar{2}]$, with the ``$\bar{1}$'' and ``$\bar{4}$'' in positions $p-1$ and $p$.  The gap can be filled to yield one of three patterns:
\[
  [\bar{1}\;\bar{4}\;\bar{2}\;3],\quad [\bar{3}\;\bar{1}\;\bar{4}\;\bar{2}], \quad \text{or} \quad [3\;\bar{1}\;\bar{4}\;\bar{2}].
\]
In each case, right-multiplication by $s_m$ yields a forbidden pattern, verifying that $ws_m$ is not vexillary.

Finally, suppose $m=0$ occurs as a label in a row with more than one box.  We may assume the label is not repeated, so this must be the last row of $Y$, i.e., $p_s=1$ and $q_s>1$.  From the construction, this means the positive integer $1$ occurs to the right of position $1$ in $w$.  Applying $s_0$ therefore results in the forbidden pattern $[2\;1]$.
\end{proof}

By reversing the rules for removing a corner, one obtains rules for adding a box to a labelled Young diagram.  Suppose $Y$ has labels $m_i$ for corners in rows $k_i$, and fix an index $j$ such that $m_j-m_{j+1}>1$.  
(To include extreme cases $j=0$ and $j=s$, we allow $0\leq j\leq s$ and use the conventions $k_0=0$, $m_0=+\infty$, $\lambda_0 = +\infty$, $k_{s+1}=k_s+1$, $m_{s+1}=-1$, and $\lambda_{k_{s+1}}=0$.)  
A new labelled Young diagram $Y\cup m$ is defined as follows, for certain $m$ to be specified.

\begin{enumerate}
\item If $k_{j+1}-k_{j}>1$ and $\lambda_{k_{j}}-\lambda_{k_{j+1}}>k_{j+1}-k_{j}+1$, then $Y\cup m$ is defined for any $m$ satisfying
\begin{align*}
    m_j &> m > m_{j+1}, \\
   m-m_{j+1}&\leq k_{j+1}-k_j, \qquad\text{ and } \\
  m_j-m &\leq \lambda_{k_j}-\lambda_{k_{j+1}}+k_j-k_{j+1},
\end{align*}
by creating a new corner labelled $m$ in row $k_j+1$ and leaving all other labels unchanged.

\smallskip

\item If $k_{j+1}-k_j=1$ and $\lambda_{k_{j}}-\lambda_{k_{j+1}}>2$, then $Y\cup m$ is defined for $m=m_{j+1}+1$, by placing a new box labelled $m$ in row $k_j+1$ and erasing the old label $m_{j+1}$.

\smallskip

\item If $k_{j+1}-k_j>1$ and $\lambda_{k_j}-\lambda_{k_{j+1}}=k_{j+1}-k_j+1$, then $Y\cup m$ is defined for $m=m_{j}-1$, by placing a new box labelled $m$ in row $k_j+1$ and erasing the old label $m_j$.

\smallskip

\item If $k_{j+1}-k_j=1$ and $\lambda_{k_j}-\lambda_{k_{j+1}}=2$, then $Y\cup m$ is defined for $m=m_{j+1}+1=m_j-1$, by placing a new box labelled $m$ in row $k_j+1$ and erasing the old labels $m_j$ and $m_{j+1}$.
\end{enumerate}

\noindent
An integer $m$ is \define{insertable} in a labelled Young diagram $Y$ if there is a $j$ such that $m_j>m>m_{j+1}$ and one of the above four cases holds.

\begin{corollary}
Suppose $w$ is vexillary, with labelled Young diagram $Y$.  Then $ws_m$ is vexillary of length $\ell(w)+1$ if and only if $m$ is insertable in $Y$, in which case its labelled Young diagram is $Y\cup m$.
\end{corollary}

For a labelled Young diagram $Y$ with labels $m_i$ in rows $k_i$, set $l_i = \lambda_{k_i}-m_i-1$, and define $n(Y) = \max\{m_i+k_i,\,l_i+k_i\}_{1\leq i\leq s}$.

\begin{corollary}
Let $Y$ be a labelled Young diagram with corresponding vexillary signed permutation $w$.  If $n\geq n(Y)$, then there is a sequence of vexillary signed permutations $w=w_{(0)}, w_{(1)}, \ldots, w_{(r)}=w_\circ^{(n)}$, such that for each $i$ there is an $m$ with $w_{(i+1)}=w_{(i)}s_m$ and $\ell(w_{(i+1)})=\ell(w_{(i)})+1$.
\end{corollary}

\begin{proof}
Unless $Y$ has shape $\lambda=(2n-1,2n-3,\ldots,1)$ with labels $n-1,n-2,\ldots,0$ (corresponding to the vexillary signed permutation $w_\circ^{(n)}$), there is an insertable $m$ with $0\leq m\leq n-1$.  Indeed, suppose $n\geq n(Y)$, but no such insertable $m$ exists.  Then rules (i) and (ii) imply $m_1=n-1$ must appear in the first row; all the rules together then show $m_2=n-2$ appears in the second row; and continuing this way, we see $0$ must appear in the $n$th row.  A diagram with these labels has $n(Y)\leq n$ only if it is the one corresponding to $w_\circ^{(n)}$ (in which case $n(Y)=n$).
\end{proof}

The inequalities on $m_i$ imply that the labels can be replaced by $l_i = \lambda_{k_i}-m_i-1$ to form a \define{dual labelled Young diagram} $Y^*$.  If $Y$ corresponds to the triple $\triple$, then $Y^*$ corresponds to the dual triple $\triple^*$, so the corresponding vexillary signed permutations are inverses.  The analogous results for left-multiplication by $s_m$ follow from these observations.

\begin{corollary}
Let $Y$ be a labelled Young diagram corresponding to a vexillary signed permutation $w$.  Then $s_m w$ is vexillary of length $\ell(w)-1$ if and only if $m$ is a removable label of $Y^*$.  Similarly, $s_m w$ is vexillary of length $\ell(w)+1$ if and only if $m$ is insertable in $Y^*$.

If $n\geq n(Y)$, then there is a sequence of vexillary signed permutations $w=w_{(0)}, w_{(1)}, \ldots, w_{(r)}=w_\circ^{(n)}$, such that for each $i$ there is an $m$ with $w_{(i+1)}=s_m w_{(i)}$ and $\ell(w_{(i+1)})=\ell(w_{(i)})+1$.
\end{corollary}

Inserting a label is a more flexible operation than removing a label.  Starting from a given vexillary permutation $w$, there need not be a descending chain of vexillary permutations $w_{(i)}$ ending in $w_{(r)}=\mathrm{id}$, with $w_{(i+1)}=w_{(i)}s_m$ and $\ell(w_{(i+1)})=\ell(w_{(i)})-1$.  For instance, taking $\triple=(2\;3,\;2\;2,\;3\;1)$ as in Figure~\ref{f.lyd}, the only $m$ such that $\ell(ws_m)=\ell(w)-1$ is $m=1$, but this is not a removable label.

However, at least one of $Y$ or $Y^*$ always has a removable label.  Letting $\triple$ be the corresponding triple, the next statement is proved by induction on the length of $w(\triple)$, removing a label from either $Y$ or $Y^*$.

\begin{corollary}\label{c.ydl}
For $w=w(\triple)$, we have $\ell(w) = |\lambda(\triple)|$.
\end{corollary}

\noindent
(This can also be proved directly from the construction of $w(\triple)$ by counting inversions.)

\section{Transitions}\label{s.transitions}

A signed permutation $w$ is {\em maximal grassmannian} if its only descent is at $0$.  A maximal grassmannian signed permutation corresponds to a strict partition $\lambda$, by recording the absolute values of the barred entries; for example, $w=\bar{4}\;\bar{2}\;\bar{1}\;3$ corresponds to $\lambda=(4,2,1)$.  It is straightforward to check that maximal grassmannian signed permutations are vexillary: its labelled Young diagram has shape $\lambda$ and all corners labelled $0$.  For a strict partition $\lambda$, we will write $w_\lambda$ for the corresponding maximal grassmannian signed permutation.

{\em Transitions} provide a way of reducing arbitrary signed permutations to maximal grassmannian ones.  They were used by Billey to study Schubert polynomials and Stanley symmetric functions \cite{billey}.

For $i<j$, let $t_{ij}$ be the transposition exchanging positions $i$ and $j$, and for $i\leq j$ let $s_{ij}$ exchange $i$ and $\bar\jmath$.  (Thus $t_{ij}$ is the reflection in the hyperplane defined by $e_i-e_j$, and $s_{ij}$ is the reflection in the hyperplane defined by $e_i+e_j$.) 
For any signed permutation $w$, let $m$ be the last descent, and let $j$ be the largest index greater than $m$ such that $w(m)>w(j)$.  A \define{transition} of $w$ is a signed permutation $w^-$, of the same length as $w$, such that $w^- = wt_{mj} t_{im}$ for some $i<m$ or $w^- = w t_{mj} s_{im}$ for any $i$.

Definitions and properties of the type B Stanley symmetric function $H_w$ may be found in \cite{bh} or \cite{fk}.  Here we need two properties, from \cite{billey}.  First, for a strict partition $\lambda$, we have
\begin{equation}\label{e.max-grass}
 H_{w_\lambda} = P_\lambda,
\end{equation}
where the latter is the Schur $P$-function.  Second, for any $w$, with $m$ and $j$ defined as above, there is a recursive formula
\begin{equation}\label{e.trans-recursion}
  H_w = (\;2 H_{w t_{mj} s_{mm}} \;+\;)\;\sum_{w^-\neq w t_{mj} s_{mm}} H_{w^-},
\end{equation}
the sum over transitions of $w$, with the first term appearing when $\ell(w t_{mj} s_{mm})=\ell(w)$ (i.e., when this is also a transition).

\begin{lemma}\label{l.trans-vex}
Let $Y$ be a labelled Young diagram, and assume the largest label $m$ is greater than $0$.  Suppose $m=m_1=\cdots=m_r>m_{r+1}$, and let $Y^-$ be the result of replacing the $r$th corner label $m_r$ by $m-1$.  Let $w$ and $w^-$ be the corresponding vexillary signed permutations.  Then $w^-$ is the unique transition of $w$.
\end{lemma}

\begin{proof}
Let $\triple=(\bk,\bp,\bq)$ be the triple corresponding to $Y$.  From the construction, $m=p_1-1=\cdots=p_r-1$, and $j=m+k_r$.  Note that $w(j)=\bar{q}_r$.   The transposition $t_{mj}$ swaps $w(m)$ and $w(j)$; since the sequence $w(p_1),\ldots,w(j)$ is increasing, $\ell(wt_{mj}) = \ell(w)-1$.

If there is an entry to the left of position $m$ that is less than $\bar{q}_r$, let $i<m$ be the largest such index.  Then $w^-=wt_{mj}t_{im}$ is a transition of $w$, and one checks that its labelled Young diagram is $Y^-$.  Furthermore, for any $i<i'<m$, right-multiplication by $t_{i'm}$ decreases the length of $wt_{mj}$, and for any $0<i'<i$, right-multiplication by $t_{i'm}$ increases the length of $wt_{mj}$ by at least $2$; similarly, $\ell(wt_{mj}s_{i'm})\neq\ell(w)$ for any $i'$.

If there is no entry to the left of $m$ less than $\bar{q}_r$, let $i>m$ be the smallest index such that $w(i)>q_r$.  Then $w^-=wt_{mj}s_{mi}$ is a transition of $w$, with labelled Young diagram $Y^-$.  One checks as before that this is the only transition.
\end{proof}

Combining Equations~\eqref{e.max-grass} and \eqref{e.trans-recursion} with Lemma~\ref{l.trans-vex} yields a formula for the Stanley symmetric function of a vexillary signed permutation.

\begin{corollary}\label{c.stanley}
The Stanley symmetric function $H_{w(\triple)}$ is equal to the Schur $P$-function $P_{\lambda(\triple)}$.
\end{corollary}

It follows that for a vexillary signed permutation $w=w(\triple)$, the partition $\lambda^B(w)$ defined in \cite[\S5]{bl} is equal to our $\lambda(\triple)$.



\end{document}